\documentclass{amsart}
\usepackage{hyperref}

\usepackage{math}
\usepackage{youngtab,mathrsfs,amsrefs,bbold,rotating}

\newcommand\Fh{{\widehat F}}
\renewcommand\R{{\mathcal R}}
\newcommand\Ind{{\operatorname{Ind}}}
\newcommand\Res{{\operatorname{Res}}}
\newcommand\GLr[2][r]{{\mathsf{GL}_{#1}(#2)}}
\newcommand\IA{{\mathsf{IA}}}
\newcommand\triv{{\mathbb1}}
\newcommand\A{{\mathscr A}}
\newcommand\T{{\mathscr T}}
\newcommand\J{{\mathscr J}}
\renewcommand\L{{\mathscr L}}
\newcommand\M{{\mathscr M}}
\newcommand\dd[2][]{{\frac{\partial #1}{\partial x_{#2}}}}
\newcommand\redD{{\overline D}}
\newtheorem{mainthm}{Theorem}

\begin{document}
\title{Automorphisms of free groups, I}
\author{Laurent Bartholdi}
\date{June 16, 2013}
\email{laurent.bartholdi@gmail.com}
\urladdr{\texttt{http://www.uni-math.gwdg.de/laurent}}
\address{Mathematisches Institut, Georg-August Universit\"at zu G\"ottingen, Germany}

\begin{abstract}
  We describe, up to degree $n$, the Lie algebra associated with the
  automorphism group of a free group of rank $n$. We compute in
  particular the ranks of its homogeneous components, and their
  structure as modules over the linear group.

  Along the way, we infirm (but confirm a weaker form of) a conjecture
  by Andreadakis, and answer a question by
  Bryant-Gupta-Levin-Mochizuki.
\end{abstract}
\maketitle

\section{Introduction}
Let $F$ denote a free group of rank $r$. The group-theoretical
structure of the automorphism group $A$ of $F$ is probably
exceedingly difficult to describe, but $A$ may be `graded', following
Andreadakis~\cite{andreadakis:nilpotent}, into a more manageable
object. Let $F_n$ denote the $n$th term of the lower central series of
$F$, and let $A_n$ denote the kernel of the natural map
$\aut(F)\to\aut(F/F_{n+1})$. Then $A_0/A_1=\GLr\Z$, and
$A_n/A_{n+1}$ are finite-rank free $\Z$-modules; furthermore,
$[A_n,A_m]\subseteq A_{m+n}$, and therefore
\[\L=\bigoplus_{n\ge1} A_n/A_{n+1}\]
has the structure of a Lie algebra.

Let, by comparison, $\Fh$ denote the limit of the quotients
$(F/F_n)_{n\ge1}$; it is a pronilpotent group, and $\Fh/\Fh_n$ is
naturally isomorphic to $F/F_n$. Let $B$ denote the automorphism group
of $\Fh$, and let similarly $B_n$ denote the kernel of the natural map
$\aut(\Fh)\to\aut(\Fh/\Fh_{n+1})$. Then $B_0/B_1=\GLr\Z$ and
$\M=\bigoplus_{n\ge1}B_n/B_{n+1}$ is also a Lie algebra; furthermore,
$B_n/B_{n+1}$ are also finite-rank free $\Z$-modules. In contrast to
$\L$, the structure of $\M$ is well understood: it is the automorphism
group of the free $\Z$-Lie algebra $\bigoplus_{n\ge1}\Fh_n/\Fh_{n+1}$,
and its elements may be described as ``polynomial non-commutative
first-order differential operators'', that is expressions
\[\sum_{\underline i}\alpha_{\underline i}X_{i_1}\dots
X_{i_n}\frac{\partial}{\partial X_{i_0}}
\]
in the non-commuting variables $X_1,\dots,X_r$. The embedding
$F\to\Fh$ with dense image induces a natural map $\L\to\M$, which is
injective but not surjective.

The following problems appear naturally:
\begin{enumerate}
\item\label{pb:1} Describe the closure of the image of $\L$ in $\M$.
\item\label{pb:2} Relate $A_n$ to the lower central series
  $(\gamma_n(A_1))_{n\ge1}$ of $A_1$.
\item\label{pb:3} Compute the ranks of $\L_n=A_n/A_{n+1}$ and
  $\M_n=B_n/B_{n+1}$.
\end{enumerate}

Ad~\eqref{pb:1}, Andreadakis observes that $\L_1=\M_1$ and
$\L_2=\M_2$, while $\L_3\neq\M_3$ for $r\le3$.

Ad~\eqref{pb:2}, Andreadakis
conjectures~\cite{andreadakis:nilpotent}*{page~253} that
$A_n=\gamma_n(A_1)$, and proves his assertion for $r=3,n\le3$ and for
$r=2$. This is further developed by Pettet~\cite{pettet:johnson}, who
proves that $\gamma_3(A_1)$ has finite index in $A_3$ for all $r$,
building her work on Johnson's
homomorphism~\cite{johnson:quotient}. Further results have been
obtained by Satoh~\cites{satoh:newobstructions,satoh:lcs} and, in
particular, what amounts to our Theorem~\ref{thm:4} under a slightly
stronger restriction on the parameter $n$. The arguments
in~\cite{enomoto-satoh:derivationalgebra} let one deduce
Theorems~\ref{thm:2} and~\ref{thm:3} from Theorem~\ref{thm:1}.

Ad~\eqref{pb:3}, Andreadakis proves
\[\rank\M_n=\frac{r}{n+1}\sum_{d|n+1}\mu(d)r^{(n+1)/d},\]
where $\mu$ denotes the M\"obius function, and computes for $r=3$ the
ranks $\rank(\L_n)=9,18,44$ for $n=2,3,4$ respectively.
Pettet~\cite{pettet:johnson} generalizes these calculations to
\[\rank(\L_2)=\frac{r^2(r-1)}2,\qquad\rank(\L_3)=\frac{r^2(r^2-4)}3+\frac{r(r-1)}2.
\]

\subsection{Main results} In this paper, we prove:
\begin{mainthm}\label{thm:1}
  For all $r,n$ we have $\gamma_n(A_1)\le A_n$, and
  $A_n/\gamma_n(A_1)$ is a finite group. Moreover,
  \[A_n=\sqrt{\gamma_n(A_1)},\]
  that is, $A_n$ is the set of all $g\in A$ such that
  $g^k\in\gamma_n(A_1)$ for some $k\neq0$.

  On the other hand, for $r=3,n=7$ we have $A_n/\gamma_n(A_1)=\Z/3$.
\end{mainthm}
Therefore, Andreadakis's conjecture is false, but barely so.

\begin{mainthm}\label{thm:2}
  If $r\ge n>1$, then we have the rank formula
  \begin{equation}\label{eq:2}
    \rank\L_n=\frac{r}{n+1}\sum_{d|n+1}\mu(d)r^{(n+1)/d}-\frac1n\sum_{d|n}\phi(d)r^{n/d},
  \end{equation}
  where $\phi$ denotes the Euler totient function.
\end{mainthm}
As a byproduct, Andreadakis's above calculations for $r=3$ should be
corrected to $\rank(\L_4)=43$.

In studying the structure of $\L$, I found it useful to consider
$\L_n=A_n/A_{n+1}$ not merely as an abelian group, but rather as a
$\GLr\Z$-module under the conjugation action of $A_0/A_1$, and then to
appeal to the classification of $\GLr\Q$-representations by tensoring
with $\Q$.  Theorem~\ref{thm:2} is a consequence of the following
description of $\L_n\otimes\Q$ as a $\GLr\Q$-module.

We start by a $\GLr\Q$-module decomposition of $\M_n\otimes\Q$. It
turns out that $\M_n\otimes\Q$ naturally fits into an exact sequence
\[0\longrightarrow\T_n\longrightarrow\M_n\otimes\Q\overset\tr\longrightarrow\A_n\longrightarrow0,\]
whose terms we now describe. Let $\{x_1,\dots,x_r\}$ denote a basis of
$F$. The first subspace $\T_n$ consists of the $\GLr\Q$-orbit in
$\M_n\otimes\Q$ of the automorphisms
\[T_w:x_i\mapsto x_i\text{ for all }i<r,\quad x_r\mapsto x_rw
\]
for all choices of $w\in F_{n+1}\cap\langle x_1,\dots,x_{r-1}\rangle$, and
are so called because of their affinity to `transvections'. The second
subspace $\A_n$ may be identified with the $\GLr\Q$-orbit in
$\M_n\otimes\Q$ of
\[A_{a_1\dots a_n}:x_i\mapsto x_i[x_i,a_1,\dots,a_n]\text{ for all }i,
\]
for all choices of $a_1,\dots,a_n\in F$; here and below $[u,v]$
denotes the commutator $u^{-1}v^{-1}uv$, and $[u_1,\dots,u_n]$ denotes
the left-normed iterated commutator $[[u_1,\dots,u_{n-1}],u_n]$.

For $r\ge n$, the space $\A_n$ is $r^n$-dimensional, and is isomorphic
qua $\GLr\Q$-module with $H_1(F,\Q)^{\otimes n}$, via $A_{a_1\dots
  a_n}\leftrightarrow a_1\otimes\dots\otimes a_n$; hence the name
reminding the $A_{a_1\dots a_n}$ of their `associative' origin.

Again for $r\ge n$, we may \emph{define} $\A_n$ as $H_1(F,\Q)^{\otimes
  n}$, and then the `trace map' $\tr:\M_n\to\A_n$ sends an automorphism
to the trace of its Jacobian matrix;
compare~\cite{morita:abelianquotients}.
\begin{mainthm}\label{thm:4}
  Assume $r\ge n>1$, and identify $\A_n$ with $H_1(F,\Q)^{\otimes
    n}$. Let $\Z/n=\langle\gamma\rangle$ act on $\A_n$ by cyclic
  permutation: $(a_1\otimes\cdots\otimes
  a_n)\gamma=a_2\otimes\cdots\otimes a_n\otimes a_1$.

  Then $\L_n\otimes\Q$ contains $\T_n$, and its image in $\A_n$ is the
  subspace of ``cyclically balanced'' elements $\A_n(1-\gamma)$
  spanned by all $a_1\otimes\cdots\otimes a_n-a_2\otimes\cdots\otimes
  a_n\otimes a_1$.
\end{mainthm}

The $\GLr\Q$-decomposition of $V_n=H_1(F,\Q)^{\otimes n}$ mimicks that
of the regular representation of the symmetric group $\sym n$, and is
well described through Young diagrams (see~\S\ref{ss:gln} for the
definitions of Young diagram, tableaux and major index). For example,
the decomposition of $V_n$ in irreducibles is given by all standard
tableaux with $n$ boxes.  Lie elements in $V_n$, which correspond to
inner automorphisms in $\A_n$, correspond to standard tableaux with
major index $\equiv1\pmod n$, as shown by
Klyashko~\cite{klyashko:lie}. We show:
\begin{mainthm}\label{thm:3}
  If $r\ge n$, the decomposition of $\L_n\otimes\Q$ in
  irreducibles is given as follows:
  \begin{itemize}
  \item all standard tableaux with $n+1$ boxes, major index $\equiv
    1\pmod{n+1}$, and at most $r-1$ rows, to which a column of length
    $r-1$ is added at the right;
  \item all standard tableaux with $n$ boxes, at most $r$ rows, and
    major index $\not\equiv 0\pmod n$.
  \end{itemize}
\end{mainthm}
The first class corresponds to $\T_n$, and the second one to
$\A_n$.

In fact, numerical experiments show that Theorems~\ref{thm:2}
and~\ref{thm:3} should remain true under the weaker condition $r\ge
n-1$. Illustrations appear in~\S\ref{ss:examples}.

\subsection{Main points} The proofs of
Theorems~\ref{thm:1},\ref{thm:2},\ref{thm:3} follow from classical
results in the representation theory of $\GLr\Q$. The proof of
Theorem~\ref{thm:4} uses results of Birman and
Bryant-Gupta-Levin-Mochizuki to the respective effects that a
endomorphism is invertible if and only if its Jacobian matrix is
invertible, and that in that case the trace of its Jacobian matrix is
cyclically balanced.

In fact, these last authors ask whether that condition is sufficient
for an endomorphism to be invertible; I give in~\S\ref{ss:l} an
example showing that it is not so.

\subsection{Plan}
\S\ref{ss:gln} briefly summarizes the representation theory of
$\GLr\Q$.

\S\ref{ss:fg} recalls some facts about the automorphism group of a
free group in the language of representation theory and free
differential calculus.

\S\ref{ss:fdc} recalls elementary properties of free differential
calculus.

\S\ref{ss:m} and \S\ref{ss:l} describe the Lie algebras $\M$ and $\L$
respectively, both as algebras and as $\GLr\Q$-modules.

\S\ref{ss:proofs} proves the theorems stated above.

Finally, \S\ref{ss:examples} provides some examples and illustrations
of the main results.  Depending on the reader's familiarity with the
subject, she/he may skip to~\S\ref{ss:m}.

\subsection{Thanks}
I greatly benefited from discussions with Andr\'e Henriques, Joel
Kamnitzer and Chenchang Zhu, and wish to thank them for their patience
and generosity. I am also grateful to Steve Donkin, Donna Testerman,
Takao Satoh and Naoya Enomoto for remarks and references that improved an earlier
version of the text, and to the anonymous referee for his/her valuable
remarks.

Some decompositions were checked using the computer software system
\textsc{GAP}~\cite{gap4:manual}, and in particular its implementation
of the ``meataxe''. Extensive calculations led to the second statement
of Theorem~\ref{thm:1}.

There has been a big gap between the beginning and the end of my
writing this text, and I am very grateful to Benson Farb for having
(1) encouraged me to finish the writeup (2) given me the opportunity
of doing it at the University of Chicago in a friendly and stimulating
atmosphere.

\section{$\GLr\Q$-modules}\label{ss:gln}
Throughout this \S\ we denote by $V$ the natural $\GLr\Q$-module
$\Q^r$.

We consider only algebraic representations, i.e.\ those linear
representations whose matrix entries are polynomial functions of the
matrix entries of $\GLr\Q$. The \emph{degree} of such a representation
is the degree of these polynomial functions. If $W$ is a
representation of degree $n$, then the scalar matrix $\mu\triv$ acts
by $\mu^n$ on $W$.

A fundamental construction by Weyl (see~\cite{fulton-h:rt}*{\S~15.3})
is as follows. The tensor algebra of $V$ decomposes as
\[T(V)=\bigoplus_{n\ge0}V^{\otimes n}=\bigoplus_{\lambda\text{ partition of }
  n}U_\lambda\otimes W_\lambda,
\]
where $U_\lambda$ and $W_\lambda$ are respectively irreducible $\sym
n$- and $\GLr\Q$-modules. Each irreducible $\sym n$-representation
appears exactly once in this construction, and those $W_\lambda$ which
are non-zero, i.e.\ for which $\lambda$ has at most $r$ lines,
describe all representations of $\GLr\Q$ exactly once, up to tensoring
with a power of the one-dimensional determinant representation.

Therefore, degree-$n$ representations of $\GLr\Q$ are indexed by
irreducible representations of $\sym n$, i.e.\ by conjugacy classes of
$\sym n$, i.e.\ by partitions of $\{1,\dots,n\}$, the parts
corresponding to cycle lengths in the conjugacy class. Partitions with
more than $r$ parts yield $W_\lambda=0$, and therefore do not appear
in the decomposition of $V^{\otimes n}$.

It is convenient to represent partitions as Young diagrams, i.e.\
diagrams of boxes. The lengths of the rows, assumed to be weakly
decreasing, give the parts in a partition. Thus
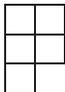
\[\yng(2,2,1)\]
is the partition $5=2+2+1$. The natural representation $V$ is
described by a single box, and its symmetric and exterior powers are
represented by a single row and a single column of boxes respectively.
A \emph{standard tableau} with shape $\lambda$, for $\lambda$ a
partition of $n$, is a filling-in of the Young diagram of $\lambda$
with each one of the numbers $\{1,\dots,n\}$ in such a way that rows
and columns are strictly increasing rightwards and downwards
respectively. For example,
\[\young(13,25,4),\quad\young(12,35,4),\quad\young(12,34,5),\quad\young(14,25,3),\quad\young(13,24,5)\]
are the standard tableaux with shape $2+2+1$.  For $\lambda$ a partition
of $n$, the multiplicity of the $\GLr\Q$-module $W_\lambda$ in
$V^{\otimes n}$ is the dimension of $U_\lambda$, and is the number of
standard tableaux with shape $\lambda$. The module $W_\lambda$ may be
written as $V^{\otimes n}c_\lambda$ for some idempotent
$c_\lambda:\Q\sym n\to\Q\sym n$, called the \emph{Schur
  symmetrizer}. This amounts to writing
\begin{equation}\label{eq:schur}
  W_\lambda=V^{\otimes n}\otimes_{\Q\sym n}U_\lambda,
\end{equation}
and $c_\lambda$ for the projection from $\Q\sym n$ to $U_\lambda$.

The \emph{major index} of a tableau $T$ is the sum of those entries
$j\in T$ such that $j+1$ lies on a lower row that $j$ in $T$. For
example, the major indices in the example above are respectively $4,5,6,7,8$.

The tensor algebra $T(V)$ contains as a homogeneous subspace the Lie
algebra generated by $V$; it is isomorphic to the free Lie algebra
$\L(F)$. The homogeneous components $\L_n(F)$ are naturally
$\GLr\Q$-modules, and their decomposition in irreducibles is described
by Klyashko~\cite{klyashko:lie}:
\begin{prop}
  The decomposition in irreducibles of $\L_n(F)$ is given by those
  tableaux with $n$ boxes whose major index is $\equiv 1\pmod n$.
\end{prop}

\subsection{Inflation} We shall use a construction of $\GLr\Q$-modules
from $\GLr[r-1]\Q$-modules, called \emph{inflation}. Consider a
$\GLr[r-1]\Q$-module $S$. It is naturally a
$\Q^{r-1}\rtimes\GLr[r-1]\Q$-module, via the projection
$\Q^{r-1}\rtimes\GLr[r-1]\Q\to\GLr[r-1]\Q$. We may embed the affine
group $\mathsf{Aff}:=\Q^{r-1}\rtimes\GLr[r-1]\Q$ in $\GLr\Q$ as the
matrices with last row $(0,\dots,0,1)$. For an algebraic group $G$,
let $\mathcal P(G)$ denote the Hopf algebra of polynomial functions on
$G$. We may then induce $S$ to a $\GLr\Q$-module $\widetilde
S:=S\otimes_{\mathcal P(\mathsf{Aff})}\mathcal P(\GLr\Q)$. In fact,
the inverse operation is easier to describe: restrict the
$\GLr\Q$-module $\widetilde S$ to $\mathsf{Aff}$, and consider then
the fixed points $S$ of $\Q^{r-1}$; this is an irreducible
$\GLr[r-1]\Q$-module.
See~\cite{jantzen:repsalgebraicgroups}*{II.2.11} for details.

This inflated module $\widetilde S$ has the same degree as $S$, and
moreover its decomposition in irreducibles admits the same Young
diagrams as $S$'s: indeed it is immediate to check that
$\Q^{r-1}\otimes_{\mathcal P(\mathsf{Aff})}\mathcal P(\GLr\Q)=\Q^r$.
Every irreducible submodule of $S$ may be seen as a submodule of
$(\Q^{r-1})^{\otimes n}$ for some $n$, using~\eqref{eq:schur}. We may
then describe $S$ as $(\Q^{r-1})^{\otimes n}\otimes_{\Q\sym n}U$ for
some $\sym n$-module $U$. We get
\[\widetilde S = (\Q^r)^{\otimes n}\otimes_{\Q\sym n}U.\]

\section{Free groups, their Lie algebras, and their automorphisms}\label{ss:fg}
Let $G$ be a group. We recall a standard construction due to
Magnus~\cite{magnus:lie}. Let $(G_n)_{n\ge1}$ be a chain of normal
subgroups of $G$, with $G_{n+1}\subseteq G_n$ and $[G_m,G_n]\subseteq
G_{m+n}$ for all $m,n\ge1$.

\begin{defn}
  The Lie ring associated with the series $(G_n)$ is
  \[\L=\L(G) = \bigoplus_{n=1}^\infty \L_n,\]
  with $\L_n=G_n/G_{n+1}$.
\end{defn}
Addition within the homogeneous component $\L_n$ is inherited from
group multiplication in $G_n$, and the Lie bracket on $\L$ is defined
among homogeneous elements by
\[\L_m\times\L_n\to\L_{m+n},\qquad (uG_{m+1},vG_{n+1})\mapsto [u,v]G_{m+n+1}.\]

A typical example is obtained by letting $(G_n)$ be the lower central
series $(\gamma_n(G))$ of $G$, defined by $\gamma_1(G)=G$ and
$\gamma_{n+1}(G)=[\gamma_n(G),G]$. The subgroups $A_n$ described in
the introduction yield an interesting (sometimes different) series.

We have an action of $G$ on $G_n$ by conjugation, which factors to an
action of $G/G_1$ on $\L$ since $G_1$ acts trivially on $G_n/G_{n+1}$.

Conversely, if $G_n$ is characteristic in $G$ for all $n$, then we may
set $H=G\rtimes\aut G$ the \emph{holomorph} of $G$, and consider the
sequence $(G_n)_{n\ge1}$ as sitting inside $H$. The resulting Lie
algebra $\L$ admits, by the above, a linear action of $\aut(G)$. The
\emph{IA-automorphisms} of $G$ --- those automorphisms that act
trivially on $G/[G,G]$ --- act trivially on $\L$ because $[G,G]\subset
G_2$, so the linear group $\GLr[]{H_1G}=\aut(G)/\IA(G)$ acts on $\L$.

Lubotzky originally suggested to me that the structure of the groups
$A=\aut(F)$ and $B=\aut(\Fh)$ could be understood by considering their
Lie algebras with $\GLr\Z$-action;
see~\cite{bass-lubotzky:linear-central}. For fruitful developments of
this idea see~\cite{church-farb:stability}.

\subsection{Pronilpotent groups}
A pronilpotent group is a limit of nilpotent groups. We recall some
useful facts gleaned from~\cite{andreadakis:nilpotent}. Let $F$ denote
a (usual) free group of rank $r$. Give $F$ a topology by choosing as
basis of open neighbourhoods of the identity the collection of
subgroups $F_n$ in $F$'s lower central series, and let $\Fh$ be the
completion of $F$ in this topology. We naturally view $F$ as a dense
subgroup of $\Fh$.

In considering series $(\Fh_n)$ of subgroups of $\Fh$, we further
require that the $\Fh_n$ be closed in $\Fh$.  Let $(\Fh_n)$ be the
(closed) lower central series of $\Fh$, defined by
$\Fh_{n+1}=\overline{[\Fh_n,\Fh]}$. We have $F_n=F\cap \Fh_n$, and
$F_n$ is dense in $\Fh_n$.  Therefore $\L_n(\Fh)=\L_n(F)$, and
by~\cite{magnus-k-s:cgt}*{Chapter~5} the module $\L_n(\Fh)$ is
$\Z$-free, of rank $r_n$ given by Witt's formula
\[r_n=\frac1n\sum_{d|n}\mu(d)r^{n/d},\]
where $\mu$ denotes the Moebius function.

From now on, we reserve the symbol $\J$ for this Lie algebra
$\L(\Fh)$. It is naturally equipped with a $\GLr\Z$-action.

\subsection{Automorphisms}
We next turn to the group $B$ of continuous automorphisms of $\Fh$, in
the usual compact-open topology. Since $F$ is dense in $\Fh$, every
automorphism $\phi\in B$ is determined by the images
$x_1^\phi,\dots,x_r^\phi\in\Fh$ of a basis of $F$. The images of these
in $\Fh/\Fh_2=\Z^r$ determine a homomorphism $B\to\GLr\Z$, which is
onto just as for the homomorphism $A\to\GLr\Z$.

The embedding $F\mapsto\Fh$ induces an embedding $A\to B$ which, as we
shall see, does not have dense image.

Let $B_1$ denote the kernel of the map $B\twoheadrightarrow\GLr\Z$;
more generally, denote by $\pi_n$ the natural map
$\pi_n:B\to\aut(F/F_{n+1})$ for all $n\ge0$, and set $B_n=\ker\pi_n$.
this defines a series of normal subgroups
$B=B_0>B_1>\cdots$. Furthermore, by a straightforward adaptation
of~\cite{andreadakis:nilpotent}*{Theorem~1.1}, we have
$[B_m,B_n]\subset B_{m+n}$, and therefore a Lie algebra
\[\M=\bigoplus_{n\ge1}\M_n=\bigoplus_{n\ge1}B_n/B_{n+1}.\]

Every element $\phi\in B_1$ is determined by the elements
$x_1^{-1}x_1^\phi,\dots,x_r^{-1}x_r^\phi\in\Fh_2$; and, conversely,
every choice of $f_1,\dots,f_r\in\Fh_2$ determines a homomorphism
$\phi_0:F\to\Fh$ defined on the basis by $x_i^{\phi_0}=x_if_i$, and
extended multiplicatively; the so-defined map $\phi_0$ extends to a
continuous map $\phi:\Fh\to\Fh$, since $\phi^{-1}(\Fh_n)$ contains
$\Fh_n$ and the $(\Fh_n)$ are a basis for the topology on
$\Fh$. Finally, $\phi$ is onto, because $\phi\pi_n$ is onto for all
$n\ge0$; indeed, a family of elements
$\{x_1^{\phi\pi_n},\dots,x_r^{\phi\pi_n}\}$ generates the nilpotent
group $F/F_n$ if and only if it generates its abelianization $F/F_2$.

By~\cite{andreadakis:nilpotent}*{\S4}
or~\cite{lubotzky:combinatorial}*{Theorem~5.8}, the $\Z$-module $\M_n$
is free of rank $r\cdot r_{n+1}$. Furthermore, $B_0/B_1\cong\GLr\Z$,
so $\M$ is naturally equipped with a $\GLr\Z$-action.

\subsection{Structure of $\J$}
In this \S\ we set $V=H_1(\Fh,\Z)\cong\Z^r$, and study the structure
of $\J$ as a $\GLr\Z$-module. It is well-known~\cite{magnus:lie}
that, as a Lie algebra, $\J$ is the free $\Z$-algebra generated by
$V$.

$\Fh_n$ is topologically spanned by $n$-fold commutators of elements
of $\Fh$, which can be written as functions $f=f(v_1,\dots,v_n)$,
where $f$ is now also seen as a commutator expression, evaluated at
elements $v_i\in\Fh$. Of the other hand, if we consider $f$ as an
element of $\J_n=\Fh_n/\Fh_{n+1}$, these $v_i$ should actually be seen
as elements of $V=\Fh/\Fh_2$, since $f(v_1,\dots,v_n)\in\Fh_{n+1}$ as
soon as one of the $v_i$ belongs to $\Fh_2$.

The action of $\GLr\Z$ on commutator expressions $f=f(v_1,\dots,v_n)$
is diagonal: $f^\rho=f(v_1^\rho,\dots,v_n^\rho)$. This yields
immediately
\begin{lem}
  The representation $\J_n$ of $\GLr\Z$ has degree $n$.\qedhere
\end{lem}

The following result dates back to the origins of the study of free
Lie rings~\cites{brandt:lierep}, and is even implicit in Witt's work;
it appears in the language of operads
in~\cite{getzler:moduli}*{Proposition~5.3}.
\begin{thm}\label{thm:adams}
  The decomposition of $\J_n$ as a $\GLr\Z$-module is given by
  inclusion-exclusion as follows:
  \[\J_n=\frac1n\bigoplus_{d|n}\mu(d)(\psi_dV)^{\otimes n/d},\]
  where $\psi_d$ is the $d$-th Adams operation (keeping the underlying
  vector space, raising eigenvalues to the $d$th power).
\end{thm}

Although this formula is explicit and allows fast computation of
character values, it is not quite sufficient to write down $\J_n$
conveniently --- it would be better to express $\J_n$ as $V^{\otimes
  n}\otimes_{\Z\sym n}S_n$ for an appropriate $\sym n$-representation
$S_n$.

Let us assume for a moment that $K$ is a ring containing a primitive
$n$-th root of unity $\varepsilon$, and that $V$ is a free $K$-module
of rank $r$. Then by~\cite{klyashko:lie} we have
\begin{equation}\label{eq:indchi}
  S_n=\Ind_{\Z/n}^{\sym n}K_\varepsilon,
\end{equation}
where $\Z/n$ acts on $K_\varepsilon\cong K$ by multiplication by
$\varepsilon$. Furthermore, if $K$ contains $\frac1n$, Klyashko
gives an isomorphism between the functors $V\mapsto \J_n(V)$ and
\[V\mapsto C_n(V)=\operatorname{Hom}_{K[\Z/n]}(K_\varepsilon,V^{\otimes n})\simeq\setsuch{v\in
  V^{\otimes n}}{v\gamma=\varepsilon v},
\]
where $\Z/n=\langle\gamma\rangle$ acts on $V^{\otimes n}$ by
permutation of the factors.

We may not assume that $\Z$ contains $n$-th roots of unity --- it
does not; however, $S_n$ is defined over $\Z$ and may be constructed
without reference to any $\varepsilon$. Numerous
authors~\cites{kraskiewicz-w:coinvariants,jollenbeck-s:cyclic} have
studied the decomposition in irreducibles of the induction from a
cyclic subgroup of a one-dimensional representation. We reproduce it
here in our notation.
\begin{prop}[\cite{kraskiewicz-w:coinvariants}]\label{thm:kw}
  The multiplicity of the irreducible representation $U_\lambda$ in
  $S_n$ is the number of standard Young tableaux of shape $\lambda$
  and major index congruent to $1$ modulo $n$.
\end{prop}

The following result seems new, and constructs efficiently the
representation $S_n$ without appealing to $n$-th roots of unity:
\begin{prop}
  Inside $\sym n$, consider the following subgroups: a cyclic subgroup
  $\Z/n$ generated by a cycle $\gamma$ of length $n$; its automorphism
  group $(\Z/n)^*$; and its subgroup $(n/d)\Z/n$ generated by
  $\gamma^{n/d}$, isomorphic to a cyclic group of order $d$. Then
  \begin{equation}\label{eq:indsum}
    S_n=\bigoplus_{d|n}\mu(d)\Ind_{((n/d)\Z/n)\rtimes(\Z/n)^*}^{\sym n}\triv,
  \end{equation}
  where $\triv$ denotes the trivial representation.
\end{prop}
For concreteness, we may identify $\sym n$ with the symmetric group of
$\mathbb Z/n$. Then $\gamma$ is the permutation $i\mapsto i+1\pmod n$,
and $(\Z/n)^*$ is the group of permutations of the form $i\mapsto
ki\pmod n$ for all $k$ coprime to $n$.
\begin{proof}
  The proof proceeds by direct computation of the characters of the
  left- and right-hand side of~\eqref{eq:indsum}, using the
  expression~\eqref{eq:indsum}.
  
  To simplify notation, we will write $C=\Z/n$. We write elements of
  $C\rtimes C^*$ as $(m,u)$. For $m\in C$ we write $m^*=n/\gcd(m,n)$ its
  order in $C$. We enumerate $C^*=\{u_1,\dots,u_{\phi(n)}\}$.
  
  It suffices actually to prove that the inductions of $K_\varepsilon$
  and $\triv$ to $C\rtimes C^*$ are isomorphic.  Let $\alpha$ denote
  the character of $\Ind_C^{C\rtimes C^*}K_\varepsilon$; then
  \[\alpha(m,u)=\begin{cases}
    \mu(m^*)\frac{\phi(n)}{\phi(m^*)} & \text{ if }u=1,\\
    0 & \text{ otherwise},\end{cases}
  \]
  where $\phi$ denotes Euler's totient function. Indeed
  $\Ind_C^{C\rtimes C^*}\alpha(m,u)$ is a
  $\phi(n)\times\phi(n)$-monomial matrix; it is the product of a
  diagonal matrix with entries
  $\varepsilon^{\mu_1},\dots,\varepsilon^{\mu_{\phi(n)}}$ and the
  permutational matrix given by $u$'s natural action on $C^*$. This
  matrix has trace $0$ unless $u=1$, in which case its trace is
  $\phi(n)/\phi(m^*)$ the sum of all primitive $m^*$-th roots of
  unity.
  
  Let $\beta_d$ denote the character of $\Ind_{C^{n/d}\rtimes
    C^*}^{C\rtimes C^*}\triv$. Then by similar reasoning
  \[\beta(m,u)=\begin{cases}
    \gcd(\frac nd,u-1) & \text{ if } \gcd(\frac nd,u-1) | m,\\
    0 & \text{ otherwise}.\end{cases}
  \]
  The result now follows from the elementary\dots
  \begin{lem}
    For any $\ell|n$, we have
    \[\sum_{\ell|d|n}\mu(d)\frac nd = \mu(\ell)\frac{\phi(n)}{\phi(\ell)}.\]
    
    If $1<u<n$, we also have
    \[\sum_{\ell|d|n}\mu(d)\gcd\left(\tfrac nd,u-1\right)=0.\]
  \end{lem}
  \dots whose proof is immediate, by noting that the left- and right-hand
  sides are multiplicative, and agree when $n$ is a prime power.
\end{proof}

\section{Free differential calculus}\label{ss:fdc}
We recall the basic notions from~\cite{fox:fdc1}. Let again $F$ denote
a free group of rank $r$, with basis $\{x_1,\dots,x_r\}$. Define
derivations
\[\dd i:\Z F\to\Z F\]
by the rules $\dd i x_i=1$, $\dd i(x_i^{-1})=x_i^{-1}$, and $\dd i
(x_j^{\pm1})=0$ if $i\neq j$, extended to $\Z F$ linearity and by the
Leibniz rule $\dd i(uv)=\dd[u]iv^o+u\dd[v]i$, where $o:\Z F\to\Z$
denotes the augmentation map.

A simple calculation proves the formula
\begin{equation}\label{eq:diff[]}
  \dd i[u,v]=u^{-1}v^{-1}\Big((u-1)\dd[v]i-(v-1)\dd[u]i\Big).
\end{equation}
In particular, if $u\in\gamma_n(F)$, then modulo $\gamma_{n+2}(F)$ we
have
\[\dd i[u,x_i]\equiv(u-1)-(x_i-1)\dd[u]i,\qquad\dd i[u,x_j]\equiv-(x_j-1)\dd[u]i\text{ if }j\neq i.\]

Denote by $\varpi\le\Z F$ the kernel of $o$; then for all $u\in\Z F$
we have the ``fundamental relation''~\cite{fox:fdc1}*{(2.3)}
\[u-u^o=\sum_{i=1}^r\dd[u]i(x_i-1).\]

Write $X_i=x_i-1$ for $i\in\{1,\dots,r\}$, and consider the ring
\[\R=\Z\langle\!\langle X_1,\dots,X_r\rangle\!\rangle\] of
non-commutative formal power series. The map $\tau:x_i\mapsto X_i+1$
defines an embedding of $F$ in $\R$, which extends to an
embedding $\tau:\Z F\to\R$. Let $\varpi$ denote also the
fundamental ideal $\langle X_1,\dots,X_r\rangle$ of $\R$; this
should be no cause of confusion, since $\varpi^\tau=\varpi\cap(\Z
F)^\tau$. The ring $\R$ is graded, with homogeneous component
$\R_n$ of rank $r^n$, spanned by words of length $n$ in $X_1,\dots,X_r$.

The dense subalgebra of $\R$ generated by the $X_i$ is free of rank
$r$; it is therefore a Hopf algebra, isomorphic to the enveloping
algebra of the free Lie algebra $\J$. From now on, we consider $\J$ as
a Lie subalgebra of $\R$ in this manner.

\section{Structure of $\M$}\label{ss:m}
We are ready to understand the Lie algebra $\M$ associated with the
automorphism group $B$ of $\Fh$.

The module $V=H_1(\Fh;\Z)$ naturally identifies with $\Fh/\Fh_2$. Its
dual, $V^*$, identifies with homomorphisms $\Fh\to\Z$.

\begin{thm}\label{thm:LM}
  The $\GLr\Z$-module $\M_n$ is isomorphic to $V^*\otimes
  \J_{n+1}$. The isomorphism $\rho:V^*\otimes\J_{n+1}\to\M_n$ is
  defined on elementary tensors $\alpha\otimes f$ by
  \[\alpha\otimes f\mapsto\{x_i\mapsto x_if^{\alpha(x_i)}\},\]
  and extended by linearity.
\end{thm}
The proof is inspired from~\cite{mostowski:autrelfree}; see
also~\cite{lubotzky:combinatorial}*{Lemma~5.7}.
\begin{proof}
  Consider an elementary tensor $\alpha\otimes f$. There is a unique
  endomorphism $\phi:\Fh\to\Fh$ satisfying
  $x_i^\phi=x_if^{\alpha(x_i)}$, so $\rho$'s image in contained in
  $B$. Next, $\{x_if^{\alpha(x_i)}\}$ is a basis of $\Fh$, since it
  spans $\Fh/\Fh_2$, so $\phi$ is invertible.

  The map $\rho$ is well-defined: if $f\in F_{n+2}$, then the
  automorphism $x_i\mapsto x_if^{\alpha(x_i)}$ of $\Fh$ belongs to
  $B_{n+1}$, so the automorphism $\alpha\otimes f$ is may be defined
  indifferently for an element $f\in\J_{n+1}$ or its representative
  $f\in F_{n+1}$.

  Let us denote temporarily by $x_1^*,\dots,x_r^*$ the dual basis of
  $V^*$, defined by $x_i^*(x_j)=\triv_{ij}$.

  We construct a map $\sigma:B_n\to V^*\otimes \Fh_{n+1}$. Let
  $\phi\in B_n$ be given.  Then $x_i^\phi\equiv x_i\mod \Fh_{n+1}$, so
  $x_i^\phi=x_if_i$ for some $f_i\in \Fh_{n+1}$. We set
  \[\phi^\sigma=\sum_{i=1}^rx_i^*\otimes f_i.\]
  If $\phi\in B_{n+1}$, then $f_i\in \Fh_{n+2}$, so $\phi^\sigma\in
  V^*\otimes \Fh_{n+2}$. It follows that $\sigma$ induces a well-defined
  map $\M_n\to V^*\otimes\J_{n+1}$. Furthermore the maps $\rho$ and
  $\sigma$ are inverses of each other.

  Next, we check that $\rho$ is linear. Consider $(\alpha\otimes
  f)^\rho=\phi$ and $(\beta\otimes g)^\rho=\chi$. Then
  \[x^{\phi\chi}=(xf^{\alpha(x)})^\chi=(xf^{\alpha(x)})g^{\beta(xf^{\alpha(x)})}
  =xf^{\alpha(x)}g^{\beta(x)}=x^{\rho(\alpha\otimes f+\beta\otimes g)}.\]

  Finally, we show that the $\GLr\Z$-actions are compatible. Choose
  an element of $\GLr\Z$, and lift it to some $\mu\in B$. Consider
  $(\alpha\otimes f)^\rho=\phi$. Then $(\alpha\otimes
  f)^\mu=\alpha'\otimes f^\mu$, where $\alpha'\in V^*$ is defined by
  $\alpha'(x)=\alpha(x^{\mu^{-1}})$, so
  \[x^{\phi^\mu}=x^{\mu^{-1}\phi\mu}=\left(x^{\mu^{-1}}f^{\alpha(x^{\mu^{-1}})}
  \right)^\mu=x(f^\mu)^{\alpha(x^{\mu^{-1}})}=x^{(\alpha'\otimes f^\mu)^\rho}.\qedhere\]
\end{proof}

The Lie bracket on $\M$ may be expressed via the identification $\rho$
of Theorem~\ref{thm:LM}.
\begin{thm}\label{thm:bracketM}
  Consider $\alpha,\beta\in V^*$ and $f=f(v_0,\dots,v_m)\in \J_{m+1}$
  and $g=g(w_0,\dots w_n)\in \J_{n+1}$. Then the bracket $\M_m\times
  \M_n\to \M_{m+n}$ is given by
  \begin{multline*}
    [\alpha\otimes f,\beta\otimes g]=
    \alpha\otimes\sum_{i=0}^m\beta(v_i)f(v_0,\dots,v_{i-1},g,v_{i+1},\dots,v_m)\\
    {} - \beta\otimes\sum_{i=0}^n\alpha(w_i)g(w_0,\dots,w_{i-1},f,w_{i+1},\dots,w_n).
  \end{multline*}
\end{thm}
\begin{proof}
  Write $(\alpha\otimes f)^\rho=\phi$ and $(\beta\otimes
  g)^\rho=\chi$. Then $\phi^{-1}=(\alpha\otimes f^{-\phi^{-1}})^\rho$ and
  $\chi^{-1}=(\beta\otimes f^{-\chi^{-1}})^\rho$; indeed
  \[(x^\phi)^{\phi^{-1}}=(xf^{\alpha(x)})^{\phi^{-1}}
  =xf^{-\phi^{-1}\alpha(x)}f^{\alpha(x)\phi^{-1}}=x.\]
  We then compute
  \begin{align*}
    x^{[\phi,\chi]}&= x^{\phi^{-1}\chi^{-1}\phi\chi} = (xf^{-\phi^{-1}\alpha(x)})^{\chi^{-1}\phi\chi}\\
    &= (xg^{-\chi^{-1}\beta(x)})^{\phi\chi}f^{-\alpha(x)\phi^{-1}\chi^{-1}\phi\chi}\\
    &= (xf^{\alpha(x)})^\chi g^{-\beta(x)\chi^{-1}\phi\chi}f^{-\alpha(x)\phi^{-1}\chi^{-1}\phi\chi}\\
    &= xg^{\beta(x)}f^{\alpha(x)\chi} g^{-\beta(x)\chi^{-1}\phi\chi}f^{-\alpha(x)\phi^{-1}\chi^{-1}\phi\chi}\\
    &\equiv x\left(f^{\alpha(x)\chi}f^{-\alpha(x)\phi^{-1}\chi^{-1}\phi\chi}\right)\left(g^{\beta(x)}g^{-\beta(x)\chi^{-1}\phi\chi}\right)\mod \Fh_{n+m+2}\\
    &\equiv xf^{\alpha(x)(\chi-1)}g^{\beta(x)(1-\phi)}=x^{(\alpha\otimes
      f^{\chi-1}-\beta\otimes g^{\phi-1})^\rho}\mod \Fh_{n+m+2}.
  \end{align*}
  Now, again computing modulo $\Fh_{n+m+2}$, we have
  \[f^\chi=f(v_0g^{\beta(v_0)},\dots,v_mg^{\beta(v_m)})\equiv
  f\prod_{i=0}^m f(v_0,\dots,v_{i-1},g,v_{i+1},\dots,v_m)^{\beta(v_i)},\]
  and similarly for $g$, so the proof is finished.
\end{proof}

In fact, the dual basis $\{x_i^*\}$ of $V^*$ is naturally written
$\{\dd i\}$; in that language, Theorem~\ref{thm:LM} can be rephrased
in an isomorphism
\[\rho: \sum_{i=1}^rf_i\dd i\mapsto\bigg(\phi:x_j\mapsto
x_j\prod_{i=1}^rf_i^{\dd[x_j]i}=x_jf_j\bigg)\] between $\M$ and
order-$1$ differential operators on $\R$. Furthermore, if $\sum f_i\dd
i\in\M_n$ then $f_i\in\Fh_{n+1}$, so $f_i-1\in\mathcal
R_{n+1}$. Theorem~\ref{thm:bracketM} then expresses the Lie bracket on
$\M$ as a kind of ``Poisson bracket'': for $Y\in\R_{n+1}$ and
$Z\in\R_{m+1}$, we have
\[\left[Y\dd i,Z\dd j\right]=\dd[Y]jZ\dd i-\dd[Z]iY\dd j\in
V^*\otimes\R_{n+m+1}.\]

The representation $\M_n$ can also be written in terms of
representations of the symmetric group, as follows.  The
representation $\M_n\otimes\det$ has degree $n+r$, and therefore may be
written as $V^{\otimes(n+r)}\otimes_{\Z\sym{n+r}}T_n$, for some
representation $T_n$ of $\sym{n+r}$. Recall that $S_n$ denotes the
representation of $\sym n$ corresponding to the Lie submodule
$\J_n\subset\R_n$.
\begin{prop}\label{prop:Tn}
  Let $(-1)$ denote the sign representation of $\sym{r-1}$. Then
  \[T_n=\Ind_{\sym{n+1}\times\sym{r-1}}^{\sym{n+r}}S_{n+1}\otimes(-1).\]
\end{prop}
\begin{proof}
  Let $W,W'$ be two $\GLr\Z$-representations, of degrees $m,m'$
  respectively. Then they may be written $W=V^{\otimes
    m}\otimes_{\Z\sym m}T$ and $W'=V^{\otimes
    m'}\otimes_{\Z\sym{m'}}T'$ for representations $T,T'$ of $\sym
  m,\sym{m'}$ respectively. Their tensor product $W\otimes W'$ then satisfies
  \[W\otimes W'\cong
  V^{\otimes(m+m')}\otimes_{\Z\sym{m+m'}}\Ind_{\sym
    m\times\sym{m'}}^{\sym{m+m'}}T\otimes T'.\] The proposition then
  follows from Theorem~\ref{thm:LM}, with $W=V^*$ and $W'=\J_{n+1}$,
  since $V^*\otimes\det=V^{\otimes{r-1}}\otimes_{\Z\sym{r-1}}(-1)$.
\end{proof}

\subsection{Decomposition in $\GLr\Q$-modules}
We now turn to the fundamental decomposition of the module $\M_n$. Its
$\GLr\Z$-module structure seems quite complicated; so we content
ourselves with a study of the $\GLr\Q$-module $\M_n\otimes\Q$. We
define the following two submodules of $\M_n\otimes\Q$. The first,
$\T_n$, is spanned by the $\GLr\Q$-orbit of the automorphisms
\[T_w:x_i\mapsto x_i\text{ for all }i<r,\quad x_r\mapsto x_rw
\]
for all choices of $w\in F_{n+1}\cap\langle x_1,\dots,x_{r-1}\rangle$.
The second subspace, $\A_n$, is spanned by the automorphisms
\[A_{a_1\dots a_n}:x_i\mapsto x_i[x_i,a_1,\dots,a_n]\text{ for all }i,
\]
for all choices of $a_1,\dots,a_n\in F$.

\begin{lem}\label{lem:split}
  We have $\M_n\otimes\Q=\T_n\oplus\A_n$ qua $\GLr\Q$-modules.
\end{lem}
\begin{proof}
  Via Theorem~\ref{thm:LM}, we may view $\T_n$ and $\A_n$ as
  submodules of $V^*\otimes\J_{n+1}$. Then $\T_n$ is spanned by all
  those $\alpha\otimes f(v_0,\dots,v_n)$ such that $\alpha(v_i)=0$ for
  all $i\in\{0,\dots,n\}$, when $f$ ranges over $n$-fold
  commutators. On the other hand, $\A_n$ is spanned by the $\sum_i
  x_i^*\otimes f(x_i,v_1,\dots,v_n)$ when $f$ ranges over $n$-fold
  commutators. We conclude that $\T_n\cap\A_n=0$, and it remains to
  check, by dimension counting, that $\T_n+\A_n=\M_n\otimes\Q$.

  By the Littlewood-Richardson rule, the module $T_n$ from
  Proposition~\ref{prop:Tn} is a sum of irreducible representations of
  $\sym{n+r}$ of all possible skew shapes $\lambda$ obtained by
  playing the ``jeu du taquin'' on a column of height $r-1$ (the Young
  diagram of the sign representation) and shapes $\mu$ appearing in
  $S_{n+1}$.

  In $\lambda$, the column of height $r-1$ occupies either the places
  $(1,1),\dots,(r-1,1)$, or the places $(2,1),\dots,(r,1)$. We shall
  see that the first case corresponds to summands of $\T_n$, and the
  second case corresponds to summands of $\A_n$.

  In the first case, the original representation $\mu$ of $\sym{n+1}$
  subsists, on the condition that it contains at most $r-1$
  lines. These summands therefore precisely describe those
  representations of $\sym{n+1}$ on $\L_{n+1}$ that come from
  $F_{n+1}\cap\langle x_1,\dots,x_{r-1}\rangle$.

  In the second case, the ``jeu du taquin'' procedure asks us to
  remove box $(1,2)$ from $\mu$ to fill position $(1,1)$, and to
  propagate this hole in $\mu$. This amounts to restricting $S_{n+1}$
  to the natural subgroup $\sym n$ of $\sym{n+1}$. Recall that
  $S_{n+1}=\Ind_C^{\sym{n+1}}\chi$, for a primitive character $\chi$
  of the cyclic group $C$ generated by a cycle of length $n+1$. By
  Mackey's theorem,
  \[\Res_{\sym n}^{\sym{n+1}}S_{n+1}=\Ind_{C\cap\sym n}^{\sym
    n}\Res_{C\cap\sym n}^C\chi=\Ind_1^{\sym n}\triv,
  \]
  since $C\sym n=\sym{n+1}$ and $C\cap\sym n=1$. Now the
  $\GLr\Q$-representation associated with the $\sym
  n$-representation $\Q\sym n$ is the full space $V^{\otimes n}$,
  which spans $\A_n$ naturally.
\end{proof}

The correspondence $X_{i_1}\dots X_{i_n}\mapsto A_{x_{i_1}\dots
  x_{i_n}}$ defines a linear map $\theta'_n:\mathcal R_n\to\A_n$.
\begin{lem}\label{lem:theta'bij}
  If $r\ge n$ then $\theta'_n$ is bijective, and makes $\A_n$
  isomorphic to $V^{\otimes n}$ qua $\GLr\Q$-module.
\end{lem}
\begin{proof}
  It is clear that $\theta'_n$ is onto, and is compatible with the
  $\GLr\Q$-action.

  Continuing with the argument of the previous lemma, the Young
  diagrams $\lambda$ and $\mu$ automatically have at most $r$ rows
  because $r\ge n$; so, since $\Res_{\sym n}^{\sym{n+1}}S_{n+1}$ is
  the regular representation, $\A_n\cong V^{\otimes n}$ has the same
  dimension as $\R_n$, so $\theta'_n$ is injective.
\end{proof}

Note, however, that $\theta'_n$ is not injective for $n>r$, and that
the $\theta'_n$ do not assemble into an algebra homomorphism
$\R\to\M$.  There does exist, however, an algebra homomorphism
$\theta:\R\to\M$, defined as $\theta'_1$ on $V$ and extended
multiplicatively to $\R$. It gives $\M$ the ``matrix-like'' algebra
structure (compare with~\eqref{eq:jacobian})
\[\Big(Y\dd i\Big)\cdot\Big(Z\dd j\Big)=\dd[Y]jZ\dd i.
\]
There does not seem to be any simple formula for the components
$\theta_n$ of $\theta$, which are ``deformations'' of $\theta'_n$.
\begin{lem}
  $\theta$ is an algebra homomorphism $\R\to\M$, that is injective up
  to degree $r$. Its image is $\bigoplus_{n\ge 0}\A_n$.  Its
  restriction $\theta\downharpoonright\J$ is injective, and has as
  image the inner automorphisms of $\M$.
\end{lem}
\begin{proof}
  Consider the following filtration of $\R_n$:
  \[\R_n^i=\langle\text{products of elements of $\R_1$ involving $\ge
    i$ Lie brackets}\rangle.
  \]
  Then $\R_n^0=\R_n$, and $\R_n^{n-1}=\J_n$. Let
  $\overline\R_n=\bigoplus_{i=0}^{n-2}\R_n^i/R_n^{i+1}$ be the
  associated graded.  A direct calculation gives
  \[A_{a_1\dots a_m}\cdot A_{b_1\dots b_n} = A_{a_1\dots a_mb_1\dots
      b_n} + \sum_{j=1}^m A_{a_1,\dots,[a_j,b_1,\dots,b_n],\dots,a_m}.
  \]
  Therefore $\theta_n=\theta'_n$ on $\J_n$, and the associated graded
  maps $\overline{\theta'}_n$ and $\overline\theta_n$ coincide;
  therefore, by Lemma~\ref{lem:theta'bij}, the map $\theta_n$ is
  injective if $n\le r$.

  This also shows that $\A=\bigoplus_{n\ge 0}\A_n$ is
  closed under multiplication. The Lie subalgebra $\J$ of $\R$ then
  naturally corresponds under $\theta$ to the span of the
  $A_{[a_1,\dots,a_n]}$, namely to inner automorphisms, acting by
  conjugation by $[a_1,\dots,a_n]$.
\end{proof}

The following description is clear from the Young diagram
decomposition of $\T_n$ given in Lemma~\ref{lem:split}:
\begin{lem}\label{lem:Tninflated}
  The module $\T_n$ is isomorphic to the inflation of the module
  $\J_{n+1}(\langle x_1,\dots,x_{r-1}\rangle)$ from $\GLr[r-1]\Q$ to
  $\GLr\Q$.
\end{lem}

\section{Structure of $\L$}\label{ss:l}
We now turn to the Lie subalgebra $\L$ of $\M$, associated with the
automorphism group of $F$. The main tool in identifying, within
$\M_n$, those automorphisms of $\Fh$ which ``restrict'' to automorphisms
of $F$ is provided by Birman's theorem. For an endomorphism $\phi:F\to
F$, we define its \emph{Jacobian matrix} and \emph{reduced Jacobian matrix}
\begin{equation}\label{eq:jacobian}
  D\phi=\left(\dd[(x_i^\phi)]j\right)_{i,j=1}^r\in M_r(\Z F),\qquad\redD\phi=D\phi-\mathbb1.
\end{equation}

\begin{thm}[\cite{birman:inverse}]
  The map $\phi:F\to F$ is invertible if and only if its Jacobian
  matrix $D\phi$ is invertible over $\Z F$.
\end{thm}

If $\phi\in\M_n$, then $\phi$ may be written in the form $\sum
f_i\otimes\dd i$ with $f_i\in F_{n+1}$. Then $\dd[f_i]j\in\varpi^n$,
so $\redD\phi\in M_r(\varpi^n)$. By the chain rule, the Jacobian matrix
of a product of automorphisms is the product of their Jacobian
matrices. Consider automorphisms $\phi\in\M_m,\psi\in\M_n$, so that
$\redD\phi\in M_r(\varpi^m)$ and $\redD\psi\in M_r(\varpi^n)$. Then
$\redD[\phi,\psi]\in M_r(\varpi^{m+n})$, and
\[\redD[\phi,\psi]\equiv[\redD\phi,\redD\psi]\pmod{\varpi^{m+n+1}}.\]

The following result by Bryant, Gupta, Levin and Mochizuki gives a
necessary condition for invertibility, which we will show is
sufficient in many cases. Recall that $\Z F$ has an augmentation ideal
$\varpi$, and that $\varpi^n/\varpi^{n+1}$ can be naturally mapped
into $\R_n$ via $\tau:(x_{i_1}-1)\cdots(x_{i_n}-1)\mapsto
X_{i_1}\cdots X_{i_n}$. The cyclic group $\Z/n=\langle\gamma\rangle$
naturally acts on $\R_n$ by cyclic permutation of the variables:
\[(X_{i_1}\cdots X_{i_n})^\gamma=X_{i_2}\cdots X_{i_n}X_{i_1}.\]
Let $\R_n^+$ denote the subspace of ``cyclically balanced'' elements
\[\R_n^+=\R_n\cdot(1-\gamma)=\{u\in\R_n:\,u\cdot(1+\gamma+\dots+\gamma^{n-1})=0\}.\]

\begin{thm}[\cite{bryant-g-l-m:nontame}]\label{thm:bglm}
  Let $J\in M_r(\Z F)$ be such that $J-\mathbb1\in M_r(\varpi^n)$. If $J$ is
  invertible and $n\ge 2$, then
  \begin{enumerate}
  \item the trace of $J-\mathbb1$ belongs to
    $(\varpi^n\cap[\varpi,\varpi])+\varpi^{n+1}$;
  \item $\tr(J-\mathbb1)^\tau\in\R_n^+$.
  \end{enumerate}
\end{thm}

\noindent Returning to our description of automorphisms $\phi\in B$ as $\sum
f_i\dd i$, we get the
\begin{cor}\label{cor:mochizuki}
  If $n\ge2$ and $\phi=\sum f_i\dd i\in\M_n$ is in the closure of
  $\L_n$, then $\sum\dd[f_i]i\in\R_n^+$.
\end{cor}
\begin{proof}
  We have $x_i^\phi=x_if_i$, so
  $(D\phi)_{i,j}=\mathbb1_{i,j}+x_i\dd[f_i]j$, and $x_iX_{i_1}\dots
  X_{i_n}\equiv X_{i_1}\dots X_{i_n}\pmod{\varpi^{n+1}}$ for all
  $i,i_1,\dots,i_n\in\{1,\dots,r\}$. We get
  $\tr(D\phi-\mathbb1)=\sum_i\dd[f_i]i$, and we apply
  Theorem~\ref{thm:bglm}.
\end{proof}

The authors of~\cite{bryant-g-l-m:nontame} ask whether the condition
in Theorem~\ref{thm:bglm} could be sufficient for $J$ to be
invertible and therefore for an endomorphism $\phi:F\to F$ to be an
automorphism.  This is not so; for example, consider $r=2$ and $n=4$,
in which case all automorphisms in $A_4$ are interior. The map
\[\phi:x\mapsto x[[x,y],[[x,y],y]],\quad y\mapsto y\]
is an element of $B_4\setminus A_4$. However,
\[\Big(\dd{}[[x,y],[[x,y],y]]\Big)^\tau=\frac{\partial}{\partial
  X}[[X,Y],[[X,Y],Y]]=YXY^2-Y^2XY\]
is in $\R_4^+$.

On the other hand, we shall see below that the condition $r\le n$
implies the sufficiency of Theorem~\ref{thm:bglm}'s condition.

\subsection{\boldmath Generators of $A$}\label{ss:gensA}
Generators of $A_1$, and therefore of $\L$, have been identified by
Magnus. He showed in~\cite{magnus:gitter} that the following
automorphisms generate $A_1$:
\[K_{i,j,k}:\begin{cases}x_i\mapsto x_i[x_j,x_k]\\
  x_\ell\mapsto x_\ell &\text{ for all }\ell\neq i.
  \end{cases}
\]
In particular $K_{i,j}:=K_{i,i,j}$ conjugates $x_i$ by $x_j$, leaving
all other generators fixed. If we let $e_{i,j}$ denote the elementary
matrix with a `$1$' in position $(i,j)$ and zeros elsewhere, then the
Jacobian matrix of $K_{i,j}$ is readily computed:
\begin{lem}\label{lem:DK1}
  \[\redD K_{i,j,k}\equiv X_je_{i,k} - X_ke_{i,j}\pmod{\varpi^2}.\]
\end{lem}
\begin{proof}
  This follows directly from~\eqref{eq:diff[]}.
\end{proof}

\begin{lem}\label{lem:trace[]}
  For every $\phi\in\M$ with $\redD\phi=(u_{i,j})_{i,j}$ we have
  \[\tr[\redD\phi,\redD K_{i,j,k}]=[u_{k,i},X_j]-[u_{j,i},X_k].\qedhere\]
\end{lem}

We now write $\Omega=\{1,\dots,r\}^*$ as index set of a basis of $\R$,
and for $\omega=\omega_1\dots\omega_n\in\Omega$ we write
$X_\omega=X_{\omega_1}\cdots X_{\omega_n}$. We also write
`$i\in\omega$' to mean there is an index $j$ such that
$\omega_j=i$. We also write `$*$' for an element of $\R$ that we don't
want to specify, because its value will not affect further
calculations.

For $i,j\in\{1,\dots,r\}$ and $\omega=\omega_1\dots\omega_n\in\Omega$
with $n\ge2$ such that $i\neq j$ and $i\not\in\omega$, choose $k\neq
i,\omega_{n-1}$ and define inductively
\[K_{i,\omega,j}=[K_{i,\omega_1\dots\omega_{n-1},k},K_{k,\omega_n,j}].\]
\begin{lem}\label{lem:DK}
  For $j\neq i\not\in\omega$ we have
  \[\redD K_{i,\omega,j}=X_\omega e_{i,j}-X_{\omega_1\dots\omega_{n-1}j}e_{i,\omega_n}.\]
\end{lem}
\begin{proof}
  The induction starts with $n=1$, and follows from
  Lemma~\ref{lem:DK1}. Then, for $n\ge2$, choose $k$ as above and compute:
  \begin{align*}
    \redD K_{i,\omega,j} &= [\redD K_{i,\omega_1\dots\omega_{n-1},k},\redD K_{k,\omega_n,j}]\\
    &= [X_{\omega_1\dots\omega_{n-1}}e_{i,k}-*e_{i,\omega_{n-1}},X_{\omega_n}e_{k,j}-X_je_{k,\omega_n}]\\
    &= X_\omega e_{i,j}-X_{\omega_1\dots\omega_{n-1}j}e_{i,\omega_n},
  \end{align*}
  since for $s\in\{j,\omega_n\}$ and $t\in\{k,\omega_{n-1}\}$ the two
  terms $e_{i,\omega_{n-1}}e_{k,s}$ and the four terms
  $e_{k,s}e_{i,t}$ vanish.
\end{proof}

\noindent Define next, for $i\neq j\neq k\neq i$ and $i\not\in\omega$,
\[L_{i,\omega,j,k}=[K_{i,\omega_2\dots\omega_nk,j},K_{j,\omega_1,i}].\]
\begin{lem}\label{lem:DL}
  For  $i\neq j\neq k\neq i\not\in\omega$ we have
  \[\redD L_{i,\omega,j,k}-1=X_{\omega_2\dots\omega_nk\omega_1}e_{i,i}-X_{\omega k}e_{j,j}+X_{\omega j}e_{j,k}-*e_{i,\omega_1}.\]
\end{lem}
\begin{proof}
  Again this is a direct calculation, using Lemma~\ref{lem:DK}:
  \begin{align*}
    \redD L_{i,\omega,j,k} &= [\redD K_{i,\omega_2\dots\omega_nk,j},\redD K_{j,\omega_1,i}]\\
    &= [X_{\omega_2\dots\omega_nk}e_{i,j}-X_{\omega_2\dots\omega_nj}e_{i,k},X_{\omega_1}e_{j,i}-X_ie_{j,\omega_1}]\\
    &= X_{\omega_2\dots\omega_nk\omega_1}e_{i,i}-X_{\omega k}e_{j,j}+X_{\omega j}e_{j,k}-X_{\omega_2\dots\omega_nki}e_{i,\omega_1},
  \end{align*}
  since for $s\in\{j,k\}$ the two terms $e_{j,\omega_1}e_{i,s}$, and
  for $t\in\{i,\omega_1\}$ the two terms $e_{i,k}e_{j,t}$, vanish.
\end{proof}

\section{Proofs of the main theorems}\label{ss:proofs}
\subsection{Theorem~\ref{thm:1}}
We start by Theorem~\ref{thm:1} from the Introduction. Recall that for
a subgroup $H\le G$ we write $\sqrt H=\{g\in G\colon g^k\in H\text{
  for some }k\neq0\}$. It is clear that $\gamma_n(A_1)\le A_n$, since
$(A_n)$ is a central series.  Moreover, consider $\phi\in A$, written
$x_i^\phi=x_if_i$ for all $i$.  Assume $\phi^k\in A_n$ for some
$n$. Then $x_i^{\phi^k}\in x_iF_{n+1}$, so $f_i^k\in F_{n+1}$. Now
$\sqrt{F_{n+1}}=F_{n+1}$, so $f_i\in F_{n+1}$, and therefore
\[\sqrt{A_n}=A_n.\]
We now turn to prove $\sqrt{A_n}=\sqrt{\gamma_n(A_1)}$. Consider the
group ring $\Q A_1$, let $\varpi$ denote its augmentation ideal, and
set $A'_n=A_1\cap(1+\varpi^n)$. It is well-known
(see~\cite{passman:gr}*{Theorem~11.1.10}) that
$\sqrt{\gamma_n(A_1)}=A'_n$. Furthermore, $A'_n/A'_{n+1}$ are free
$\Z$-modules, and $[A'_m,A'_n]\le A'_{m+n}$ for all $m,n\ge1$, so
$\L'=\bigoplus_{n\ge1}A'_n/A'_{n+1}$ is a torsion-free Lie algebra.
Also, $A=A_1=A'_1$ and $A_2=A'_2$, and both $\L\otimes\Q$ and
$\L'\otimes\Q$ are $\GLr\Q$-modules. Furthermore, all three of
$(A_n)$, $(A'_n)$ and $(\gamma_n(A_1))$ are filtrations of $A$ with
trivial intersection, so the \emph{non-graded} $\GLr\Q$-modules
$\L\otimes\Q$ and $\L'\otimes\Q$ are isomorphic, because they are all
sums of the same irreducible components.

Now the modules $\L_n\otimes\Q$ and $\L'_n\otimes\Q$ are both
characterised, within $\L\otimes\Q$ and $\L'\otimes\Q$ respectively,
as the degree-$n$ homogeneous part (that on which the scalar matrix
$\lambda\triv\in\GLr\Q$ acts by $\lambda^n$). It follows that
$\L\otimes\Q$ and $\L'\otimes\Q$ are isomorphic qua graded algebras.

The filtrations $(A'_n)$ and $(A_n)$ have trivial intersections, and
are such that their successive quotients are free abelian. Clearly
$A'_n$ is contained in $A_n$, because $A'_n$ is characterized as the
fastest-descending normal series with torsion-free successive
quotients. Furthermore, the free abelian quotient $A_n/A_{n+1}$ is
characterized as mapping into the degree-$n$ part $\L_n\otimes\Q$ of
the $\GLr\Q$-module $\L\otimes\Q$, and so is $A'_n/A'_{n+1}$. For the
former this follows from Theorem~\ref{thm:LM}, and for the latter this
follows from its description as $n$-fold iterated commutators. The
claim $A'_n=A_n$ then follows by induction on $n$.

Next, $A_1/\gamma_n(A_1)$ is a finitely generated nilpotent subgroup,
so its torsion subgroup is finite; therefore, $\gamma_n(A_1)$ has
finite index in $A_n$.

The last statement of Theorem~\ref{thm:1} is purely computational.
Using the computer system \textsc{Gap}~\cite{gap4:manual}, I have
\begin{enumerate}
\item defined for a large prime $p$ (I chose $p=61$) a group $\tilde
  G=\langle x,y,z|x^p,y^p,z^p\rangle$;
\item constructed its maximal class-$7$ nilpotent quotient $G$, a
  finite group of order $p^{3+3+8+18+48+116+312}$;
\item (in 30 minutes) constructed the set $S_1$ of Magnus generators
  $x_{ij}$ and $x_{ijk}$ of the group of IA automorphisms of $G$;
\item (in 30 minutes) constructed the set $S_2$ of commutators of
  elements of $S_1$;
\item (in 5 hours) constructed the set $S_3$ comprised of all commutators
  $[S_1,S_2]$ and all those quotients of elements of $S_2$ that belong
  to $A_3$;
\item (in 50 hours) constructed the set $S_4$ comprised of all
  commutators $[S_1,S_3]$ and quotients of elements of $S_3$ that
  belong to $A_4$;
\item (in 2 hours) identified elements of $S_4$ with their image $T_4$ in
  the $\F$-vector space $G_4^3$, via the map
  $\phi\mapsto(x^{-1}x^\phi,y^{-1}y^\phi,z^{-1}z^\phi)$;
\item let $T_5$ denote those elements of $T_4$ that act trivially on
  $G_4/G_5$;
\item let $T_6$ denote those elements of $T_5$ that act trivially on
  $G_5/G_6$.
\end{enumerate}
The resulting vector space $T_6$ has dimension $806$, which is
therefore the dimension of $\L_6$ when $r=3$. The running times are
approximative, and the computation was performed twice on a standard
(2004) PC computer.

On the other hand, I have also computed, for all primes $p\le11$ and
all $1<n<7$, an independent set $S'_n$ in $\M_n\otimes\F$ among the
set of commutators $[S_1,S_{n-1}]$; and have let $T'_6$ denote the
span of $S'_6$. It turned out to be a vector space of dimension $805$
for $p=3$, and $806$ in all other cases.

I have then lifted a generator of $T_6\otimes\F[3]/T'_6\otimes\F[3]$
to $\M_6$, as follows:
\begin{multline*}
  \begin{aligned}
    x &\mapsto x\cdot[x,[[x,y],[x,[[y,z],z]]]]\cdot[x,[[x,y],[[x,z],[y,z]]]]^{-1},\\
    y &\mapsto y\cdot[[x,[y,[y,z]]],[x,[y,z]]]^{-1}\cdot[[x,[y,[y,[y,z]]]],[x,z]]^{-1},\\
    z &\mapsto z\cdot[[[x,z],[[x,z],[y,z]]],y]\cdot[[x,[y,[y,z]]],[[x,z],z]]\\
  \end{aligned}\\
  \cdot[[x,[y,[[y,z],z]]],[x,z]]\cdot[[[x,[y,z]],[y,z]],[x,z]]^{-1}.
\end{multline*}

This is the image in $\M_6$ of an automorphism of $F$; it does not
belong to $\gamma_6(A_1)$, but its cube does, as another lengthy
calculation shows.

\subsection{Theorem~\ref{thm:4}}
First, we show that $\T_n$ is contained in $\L_n$. For every
$w\in\langle x_1,\dots,x_{r-1}\rangle$, the endomorphism $T_w$ of $F$
is invertible (either directly, noting its inverse is $T_{w^{-1}}$, or
because its Jacobian is unipotent); if furthermore $w\in F_{n+1}$,
then $T_w$ defines an element of $\T_n\cap\L_n$. Since $\T_n$ is
generated by the $T_w$ qua $\GLr\Q$-module, we are done.

By Corollary~\ref{cor:mochizuki}, the image of $\L_n$ under the trace
map belongs to the space $\R_n^+$ of cyclically balanced elements as
soon as $n\ge2$.  We now claim that, if furthermore $n\le r$ and
$W\in\R_n^+$ is cyclically balanced, then there exists an automorphism
$\phi\in\L_n$ with trace $W$.

Using the $\GLr\Q$-action, it suffices to check this for an elementary
$W$, of the form $W=x_rU-Ux_r$ with $U\in\R_{n-1}$. In fact, by
linearity, we may even reduce ourselves to considering $U=X_\omega$ a
word in $\{X_1,\dots,X_r\}$ of length $n-1$, see the notation
of~\S\ref{ss:gensA}.

We first note that, if there is a letter $x_i$ which does not occur in
$W$, then by Lemmata~\ref{lem:DK} and~\ref{lem:trace[]} the
automorphism $K_{i,\omega,j}$ belongs to $\L$, and then
\[\tr[K_{i,\omega,j},K_{j,r,i}]=[X_\omega,X_r]=-W.\]

This shows that all cyclically balanced elements of degree $<r$ appear
as traces of automorphisms in $\L$. To complete the argument, it
suffices to consider the case $r=n$. A cyclically balanced word is
either covered by the previous case, or contains a single instance of
each letter. Using Lemma~\ref{lem:DL}, consider the expression
$\tr[L_{i,\omega,j,k},K_{k,\ell,j}]=[X_{\omega j},X_\ell]$. The
restrictions on it are that there exists $i,k\in\{1,\dots,r\}$ with
$\ell\neq j\neq k\neq i$ and $i\not\in\omega j$; all words with a
single instance of each letter are covered by these conditions.

\subsection{Theorems~\ref{thm:2} and ~\ref{thm:3}}
Let $\R_n^i$ denote the subspace of $\R_n$ spanned by Young tableaux
with major index $\equiv i\pmod n$. Then, by Theorem~\ref{thm:3}, the
rank of $\L_n$ is $\rank\M_n-\rank\R_n^0$. Imitating Klyashko's
argument~\cite{klyashko:lie}, let $\varepsilon$ denote a primitive
$n$th root of unity. Then $\R_n^i$ may be written, qua
$\GLr\Q$-module, as
\[\R_n^i\cong\R_n\otimes_{\Q\sym n}\Q\sym n\kappa_n\cong
\R_n\otimes_{\Q\sym n}\Q\sym n\theta_n,
\]
where the idempotents $\kappa_n,\theta_n\in\Q\sym n$ are given by
\[\kappa_n=\frac1n\sum_{\sigma\in\sym
  n}\varepsilon^{i\operatorname{maj}(\sigma)}\sigma,\qquad
\theta_n=\frac1n\sum_{j=1}^n\varepsilon^{ij}(1,2,\dots,n)^i.\]
We therefore get
\[\dim\R_n^i=\frac1n\sum_{j=1}^n\varepsilon^{ij}r^{(n,j)},\]
so in particular
\begin{equation}\label{eq:2:1}
  \dim\R_n^1=\frac1n\sum_{d|n}\mu(d)r^{n/d},\qquad\dim\R_n^0=\frac1n\sum\phi(d)r^{n/d}
\end{equation}
using the identities $\sum_{\gcd(n,j)=n/d}\varepsilon^j=\mu(d)$ and
$\sum_{\gcd(n,j)=n/d}1=\phi(d)$. Inserting~\eqref{eq:2:1} in
$\rank\L_n=r\rank\R_{n+1}^1-\rank\R_n^0$ yields~\eqref{eq:2}.

\section{Examples and illustrations}\label{ss:examples}
Recall from Theorem~\ref{thm:kw} that the $\GLr\Q$-module $\J_n$
decomposes as a direct sum of irreducible representations, indexed by
tableaux with $n$ boxes and major index $\equiv1\mod n$. Here are the
small-dimensional cases; recall that the $\GLr\Q$-decomposition of
$\J_n$ consists of those representations indexed by diagrams with at
most $r$ lines:
\def\tyoung(#1){{\mbox{\tiny\young(#1)}}}
\def\tyng(#1){{\mbox{\tiny\yng(#1)}}}
\def\strut{\rule{0em}{2.2ex}}
\[\begin{array}{r|c|c|c|c|c|c|}
  n & 1 & 2 & 3 & 4 & 5 & 6\\ \hline
  \strut\dim S_n & 1 & 1 & 2 & 6 & 24 & 120\\
  S_n & \tyoung(1) & \tyoung(1,2) & \tyoung(13,2) &
  \begin{matrix}\tyoung(134,2)\\ \tyoung(12,3,4)\end{matrix} &
  \begin{matrix}\tyoung(1345,2)\\ \tyoung(124,35) \\
    \tyoung(124,3,5)\\ \tyoung(12,34,5)\\ \tyoung(15,2,3,4)\end{matrix} &
  \begin{matrix}
    \begin{matrix}\tyoung(13456,2) & \tyoung(1245,36) & \tyoung(125,346)\end{matrix}\\
    \begin{matrix}\tyoung(123,46,5) & \tyoung(125,34,6) & \tyoung(146,25,3)\end{matrix}\\
    \begin{matrix}\tyoung(1236,4,5) & \tyoung(1245,3,6) & \tyoung(146,2,3,5)\end{matrix}\\
    \begin{matrix}\tyoung(13,24,5,6) & \tyoung(14,26,3,5) & \tyoung(13,2,4,5,6)\end{matrix}
  \end{matrix}\\ \hline
  \strut\dim\J_n, r=2 & 2 & 1 & 2 & 3 & 6 & 9\\
  \strut\dim\J_n, r=3 & 3 & 3 & 8 & 18 & 48 & 116\\
  \strut\dim\J_n, r=4 & 4 & 6 & 20 & 60 & 204 & 670\\ \hline
\end{array}\]

We next describe the decomposition of $\M_n$. Using
Theorem~\ref{thm:LM}, it may be computed by the Littlewood-Richardson
rule~\cite{schutzenberger:taquin}. We concentrate on small values of $r$:

\subsection{\boldmath $r=2$}
Representations of $\GLr\Q$ are classified by Young diagrams with at
most two lines. We may therefore express $\M_n$ as a combination of
submodules of shape $(a,b)$ for some $a+b=n$ with $a\ge b$.
Theorem~\ref{thm:adams} gives a simple answer for $\J_n$; the
decomposition of $\M_n$ is then obtained via Theorem~\ref{thm:LM}:
\begin{prop}
  Define the function
  \[\theta(a,b)=\frac1{a+b}\sum_{d|(a,b)}\mu(d)\binom{(a+b)/d}{a/d}.\]
  Then the multiplicity of $(a,b)$ in $\J_n$ is
  \[\chi_{(a,b)}^{\J_n} = \theta(a,b) - \theta(a-1,b+1).\]
  The total number of irreducible representations in $\J_n$ is
  \[i_n=\begin{cases} \theta(n/2,n/2) & \text{ if }n\equiv0[2],\\
    \theta((n+1)/2,(n-1)/2)=\frac1{n+1}\binom{n}{(n-1)/2} & \text{ if }n\equiv1[2].\end{cases}
  \]
  The multiplicity of $(a,b)$ in $\M_n$ is
  \[\chi_{(a,b)}^{\M_n} = \begin{cases}
    \theta(a+1,b) - \theta(a-1,b+2) & \text{ if }a>b,\\
    \theta(a+1,b) - \theta(a,b+1) & \text{ if }a=b,\end{cases}
  \]
  while the total number of irreducible representations in $\M_n$ is
  \[a_n=\begin{cases} 2\theta(n/2+1,n/2) & \text{ if }n\equiv0[2],\\
    \theta((n+1)/2,(n+1)/2)+\theta((n+3)/2,(n-1)/2) & \text{ if }n\equiv1[2].\end{cases}
  \]
\end{prop}
\begin{proof}
  These are special cases of classical formulas,
  see for instance~\cite{foulkes:liechar}. We have
  \[\chi_{(a,b)}^{\J_n}=\frac1n\sum_{d|n}\mu(d)\chi_{(a,b)}((d,\dots,d)),\]
  where $\chi_{(a,b)}(\mu)$ denotes the value of the character
  $\chi_{(a,b)}$ on a permutation of cycle type $\mu$. Using the
  Jacobi-Trudi identity, Foulkes expresses $\chi_{(a,b)}((d,\dots,d))$
  as the determinant
  \[(n/d)!\left|\begin{matrix}[a/d]!^{-1} & [(a+1)/d]!^{-1}\\ [(b-1)/d]!^{-1} & [b/d]!^{-1}\end{matrix}\right|,\]
  where $[x]!^{-1}$ is $x!^{-1}$ if $x$ is an integer, and $0$
  otherwise. All formulas for $\J_n$ follow.
  
  Since $\M_n=V^*\otimes\J_{n+1}$ by~\ref{thm:LM}, the multiplicity of
  $(a,b)$ in $\M_n$ is obtained by the ``Jeu du taquin'' procedure, as
  \[\chi_{(a,b)}^{\M_n}=\chi_{(a+1,b)}^{\J_{n+1}}+\chi_{(a,b+1)}^{\J_{n+1}},\]
  where the last summand is understood as $0$ if $a=b$. Again all
  formulas for $\M_n$ follow.
\end{proof}

The multiplicities of irreducibles of $\GLr\Q$ in $\J_n$ and $\M_n$ are
listed in the following table, for $n\le12$:
{\small
\[\begin{array}{r|cccccccccccc|}
  n       & 1 & 2 & 3 & 4 & 5 & 6 & 7 & 8 & 9 & 10 & 11 & 12\\ \hline
\strut\J_n:\quad
  (n,0)   & 1 &   &   &   &   &   &   &   &   &    &    &   \\
  (n-1,1) &   & 1 & 1 & 1 & 1 & 1 & 1 & 1 & 1 &  1 &  1 &  1\\
  (n-2,2) &   &   &   &   & 1 & 1 & 2 & 2 & 3 &  3 &  4 &  4\\
  (n-3,3) &   &   &   &   &   & 1 & 2 & 4 & 5 &  8 & 10 & 13\\
  (n-4,4) &   &   &   &   &   &   &   & 1 & 5 &  8 & 15 & 22\\
  (n-5,5) &   &   &   &   &   &   &   &   &   &  5 & 12 & 26\\
  (n-6,6) &   &   &   &   &   &   &   &   &   &    &    &  9\\
  \text{total}
          & 1 & 1 & 1 & 1 & 2 & 3 & 5 & 8 &14 & 25 & 42 & 75\\ \hline
  \strut\M_n:\quad
  (n,0)   & 1 & 1 & 1 & 1 & 1 & 1 & 1 & 1 & 1 &  1 &  1 &  1\\
  (n-1,1) &   & 1 & 1 & 2 & 2 & 3 & 3 & 4 & 4 &  5 &  5 &  6\\
  (n-2,2) &   &   &   & 1 & 2 & 4 & 6 & 8 &11 & 14 & 17 & 21\\
  (n-3,3) &   &   &   &   &   & 2 & 5 &10 &16 & 25 & 35 & 49\\
  (n-4,4) &   &   &   &   &   &   &   & 5 &13 & 37 & 48 & 77\\
  (n-5,5) &   &   &   &   &   &   &   &   &   & 12 & 35 & 77\\
  (n-6,6) &   &   &   &   &   &   &   &   &   &    &    & 33\\
  \text{total}
          & 1 & 2 & 2 & 4 & 5 &10 &15 &28 &45 & 84 &141 &264\\ \hline
  \end{array}
\]}

\subsection{\boldmath $r=3$ and $r=4$}
For $r>2$, some degree-$n$ irreducibles for $\GLr\Q$ in $\M_n$ are
represented by a Young diagram with $n$ boxes, and others are
represented as a formal quotient of a Young diagram with $n+r$ boxes
by the degree-$r$ determinant representation, following
Lemma~\ref{lem:Tninflated}. We indicate the latter representations by
putting $*$'s in the first column (of height $r-1$).
Table~\ref{table:3} describes the irreducibles of $\M_n\otimes\Q$ for
$1\le n\le 5$ and $r=3$, while Table~\ref{table:4} describes them for
$r=4$.  Whenever possible, the Young diagrams are filled in so as to
separate them according to their major index.

\begin{table}
\[\begin{array}{r|c|c|c|c|c|c}
  n & 1 & 2 & 3 & 4 & 5 & 6\\ \hline
  \strut\dim\M_n & 9 & 24 & 54 & 144 & 348 & 936\\
  \strut\dim\L_n & 9 & 18 & 43 & 120 & 297 & 806\\
  \strut\dim\T_n & 6 & 15 & 27 & 66 & 117 & 279\\
  \strut\dim\J_n & 3 & 3 & 8 & 18 & 48 & 116\\ \hline
  \begin{picture}(0,0)
    \put(-25,-130){\begin{sideways}\hbox to 160pt{\dotfill$\L_n$\dotfill}\end{sideways}}
    \put(-35,-190){\begin{sideways}\hbox to 170pt{\dotfill$\A_n$\dotfill}\end{sideways}}
  \end{picture}
  \T_n & \tyoung(*\relax,*\relax) & \tyoung(*\relax\relax,*\relax) &
  \tyoung(*\relax\relax\relax,*\relax) &
  \begin{matrix}\tyoung(*\relax\relax\relax\relax,*\relax)\\
    \tyoung(*\relax\relax\relax,*\relax\relax)\end{matrix} &
  \begin{matrix}\\[-3mm]\tyoung(*\relax\relax\relax\relax\relax,*\relax)\\
    \tyoung(*\relax\relax\relax\relax,*\relax\relax)\\
    \tyoung(*\relax\relax\relax,*\relax\relax\relax)\end{matrix}\\ \cline{2-7}
  \J_n & \tyoung(1) & \tyoung(1,2) & \tyoung(13,2) &
  \begin{matrix}\\[-2mm]\tyoung(134,2)\\ \tyoung(12,3,4)\end{matrix} &
  \begin{matrix}\tyoung(1345,2) & \tyoung(124,35)\\ \tyoung(124,3,5) & \tyoung(12,34,5) \end{matrix}\\ \cline{2-7}
  \frac{\A_n\cap\L_n}{\J_n} & & & \tyoung(12,3) &
  \begin{matrix}\\[-3mm]\tyoung(12,34)\\ \tyoung(124,3)\\\tyoung(123,4)\end{matrix} &
  \begin{matrix}\\[-3mm]\tyoung(1234,5)\quad\tyoung(1235,4)\quad\tyoung(1245,3)\\
    \tyoung(123,45)\quad\tyoung(125,34)\quad\tyoung(135,24)\\
    \tyng(3,1,1)\quad\tyng(3,1,1)\quad\tyng(2,2,1)\end{matrix}\\ \cline{2-7}
  \frac{\A_n}{\A_n\cap\L_n} & &
  \tyoung(12) & \begin{matrix}\tyoung(123)\\ \tyoung(1,2,3)\end{matrix} &
  \begin{matrix}\tyoung(1234)\\ \tyoung(13,24)\\ \tyoung(13,2,4)\end{matrix} &
  \begin{matrix}\\[-2mm]\begin{matrix}\tyoung(12345)\\ \tyoung(134,25)\end{matrix}\quad\begin{matrix}\tyoung(134,2,5)\end{matrix}\\
    \tyoung(125,3,4)\quad\tyoung(12,35,4)\end{matrix}\\ \hline
\end{array}\]
\caption{The decomposition of $\L\otimes\Q$ and $\M\otimes\Q$ for
  $r=3$}
\label{table:3}
\end{table}

Some of the Young diagrams with no $*$'s are not filled in;
Theorem~\ref{thm:3} does not apply since $n>r$ there, and the
decomposition of $\L_n$ is only partly understood. In particular, note
for $r=3,n=5$ that there are six tableaux with shape $3+1+1$, but only
five appear in the decomposition of $\M_n$, and there are five
tableaux with shape $2+2+1$, and only three appear in $\M_n$. For
$r=4,n=5$, there are four tableaux with shape $2+1+1+1$ but only three
appear in $\M_n$. It seems that the tableaux with the largest number
of rows are more likely to disappear first, as $n$ increases.

\begin{table}
\[\begin{array}{r|c|c|c|c|c|c}
  n & 1 & 2 & 3 & 4 & 5 & 6\\ \hline
  \strut\dim\M_n & 24 & 80 & 240 & 816 & 2680 & 9360\\
  \strut\dim\L_n & 24 & 70 & 216 & 746 & 2472 & 8660\\
  \strut\dim\T_n & 20 & 64 & 176 & 560 & 1660 & 5296\\
  \strut\dim\J_n & 4 & 6 & 20 & 60 & 204 & 670\\ \hline
  \begin{picture}(0,0)
    \put(-25,-175){\begin{sideways}\hbox to 210pt{\dotfill$\L_n$\dotfill}\end{sideways}}
    \put(-35,-240){\begin{sideways}\hbox to 210pt{\dotfill$\A_n$\dotfill}\end{sideways}}
  \end{picture}
  \T_n & \tyoung(*\relax,*\relax,*) & \tyoung(*\relax\relax,*\relax,*) &
  \begin{matrix}\tyoung(*\relax\relax\relax,*\relax,*)\\\tyoung(*\relax\relax,*\relax,*\relax)\end{matrix} &
  \begin{matrix}\tyoung(*\relax\relax\relax\relax,*\relax,*)\quad
    \tyoung(*\relax\relax\relax,*\relax\relax,*)\\
    \tyoung(*\relax\relax\relax,*\relax,*\relax)\quad
    \tyoung(*\relax\relax,*\relax\relax,*\relax)\end{matrix} &
  \begin{matrix}\\[-3mm]
    \tyoung(*\relax\relax\relax\relax,*\relax,*\relax)\quad
    \tyoung(*\relax\relax\relax\relax,*\relax,*\relax)\\
    \tyoung(*\relax\relax\relax\relax\relax,*\relax,*)\quad
    \tyoung(*\relax\relax\relax,*\relax\relax\relax,*)\\
    \tyoung(*\relax\relax\relax\relax,*\relax\relax,*)\quad
    3*\tyoung(*\relax\relax\relax,*\relax\relax,*\relax)
    \end{matrix}\\ \cline{2-7}
  \J_n & \tyoung(1) & \tyoung(1,2) & \tyoung(13,2) &
  \tyoung(134,2)\quad\tyoung(12,3,4) &
  \begin{matrix}\\[-3mm]\tyoung(1345,2)\quad\tyoung(124,35)\\ \tyoung(124,3,5)\quad
    \tyoung(12,34,5)\quad\tyoung(15,2,3,4)\end{matrix}\\ \cline{2-7}
  \frac{\A_n\cap\L_n}{\J_n} & & & \tyoung(12,3) &
  \begin{matrix}\tyoung(124,3)\quad\tyoung(123,4)\\
    \tyoung(12,34)\quad\tyoung(14,2,3)\quad\tyoung(1,2,3,4)\end{matrix} &
  \begin{matrix}\\[-3mm]\tyoung(1234,5)\quad\tyoung(1235,4)\quad\tyoung(1245,3)\\
    \tyoung(123,45)\quad\tyoung(125,34)\quad\tyoung(135,24)\\
    \tyoung(13,24,5)\quad\tyoung(13,25,4)\quad
    \tyoung(14,25,3)\quad\tyoung(145,2,3)\\
    \tyng(2,1,1,1)\quad\tyng(2,1,1,1)\quad
    \tyoung(135,2,4)\quad\tyoung(123,4,5)\end{matrix}\\
  \cline{2-7}
  \frac{\A_n}{\A_n\cap\L_n} & &
  \tyoung(12) & \begin{matrix}\tyoung(123)\\ \tyoung(1,2,3)\end{matrix} &
  \begin{matrix}\tyoung(1234)\\ \tyoung(13,24)\quad\tyoung(13,2,4)\end{matrix} &
  \begin{matrix}\\[-3mm]\tyoung(12345)\\ \tyoung(134,25)\quad\tyoung(134,2,5)\\
    \tyoung(125,3,4)\quad\tyoung(12,35,4)\end{matrix}\\ \hline
\end{array}\]
\caption{The decomposition of $\L\otimes\Q$ and $\M\otimes\Q$ for
  $r=4$}
\label{table:4}
\end{table}

\end{document}